\documentclass{amsart}

\newtheorem{thm}{Theorem}[section]

\newtheorem{prop}{Proposition}[section]

\newtheorem{conj}{Conjecture}[section]

\newtheorem{alphthm}{Theorem}[section]

\newcommand{\lct}{\; \raisebox{-.96ex}{$\stackrel{\textstyle <}{\sim}$} \;}
\newcommand{\gct}{\; \raisebox{-.96ex}{$\stackrel{\textstyle >}{\sim}$} \;}
\newcommand{\mbb}{\mathbb}

\begin{document}

\title[A family of fractal Fourier restriction estimates]{A family of 
       fractal Fourier restriction estimates with implications on the Kakeya 
	   problem}

\author{Bassam Shayya}
\address{Department of Mathematics\\
         American University of Beirut\\
         Beirut\\
         Lebanon}
\email{bshayya@aub.edu.lb}

\date{June 26, 2022}


\begin{abstract}
In a recent paper \cite{dz:schrodinger3}, Du and Zhang proved a fractal 
Fourier restriction estimate and used it to establish the sharp $L^2$ 
estimate on the Schr\"{o}dinger maximal function in $\mbb R^n$, $n \geq 2$. 
In this paper, we show that the Du-Zhang estimate is the endpoint of a 
family of fractal restriction estimates such that each member of the family 
(other than the original) implies a sharp Kakeya result in $\mbb R^n$ that 
is closely related to the polynomial Wolff axioms. We also prove that all 
the estimates of our family are true in $\mbb R^2$.
\end{abstract}

\maketitle

\section{Introduction}

Let $Ef=E_{\mathcal P} f$ be the extension operator associated with the unit 
paraboloid ${\mathcal P} = \{ \xi \in \mbb R^n : \xi_n = \xi_1^2 + \ldots + 
\xi_{n-1}^2 \leq 1 \}$ in $\mbb R^n$: 
\begin{displaymath}
Ef(x) = \int_{\mbb B^{n-1}} e^{-2 \pi i x \cdot (\omega, |\omega|^2)} 
        f(\omega) d\omega,	
\end{displaymath}
where $\mbb B^{n-1}$ is the unit ball in $\mbb R^{n-1}$.

Our starting point is the following fractal restriction theorem of Du and 
Zhang \cite{dz:schrodinger3}. (Throughout this paper, we denote a cube in 
$\mbb R^n$ of center $x$ and side-length $r$ by $\widetilde{B}(x,r)$.)

\begin{alphthm}[Du and Zhang {\cite[Corollary 1.6]{dz:schrodinger3}}]
\label{highfracdim}
Suppose $n \geq 2$, $1 \leq \alpha \leq n$, $R \geq 1$, 
$X= \cup_k \widetilde{B}_k$ is a union of lattice unit cubes in 
$\widetilde{B}(0,R) \subset \mbb R^n$, and
\begin{displaymath}
\gamma = \sup \frac{\#\{ \widetilde{B}_k : \widetilde{B}_k \subset 
\widetilde{B}(x',r) \}}{r^\alpha},
\end{displaymath}
where the sup is taken over all pairs 
$(x',r) \in \mbb R^n \times [1,\infty)$ satisfying 
$\widetilde{B}(x',r) \subset \widetilde{B}(0,R)$. Then to every 
$\epsilon > 0$ there is a constant $C_\epsilon$ such that
\begin{equation}
\label{latticez2}
\int_X |Ef(x)|^2 dx \leq C_\epsilon R^\epsilon \, \gamma^{2/n} 
                         R^{\alpha/n} \| f \|_{L^2(\mbb B^{n-1})}^2
\end{equation}
for all $f \in L^2(\mbb B^{n-1})$.
\end{alphthm}

In \cite{dz:schrodinger3}, Theorem \ref{highfracdim} was used to derive the
sharp $L^2$ estimate on the Schr\"{o}dinger maximal function (see
\cite[Theorem 1.3]{dz:schrodinger3} and the paragraph following the 
statement of \cite[Corollary 1.6]{dz:schrodinger3}). The authors of
\cite{dz:schrodinger3}, also used Theorem \ref{highfracdim} to obtain new
results on the Hausdorff dimension of the sets where Schr\"{o}dinger
solutions diverge (see \cite{ss:schdiv}), achieve progress on Falconer's 
distance set conjecture in geometric measure theory (see \cite{f:distsets}), 
and improve on the decay estimates of spherical means of Fourier transforms 
of measures (see \cite{tw:csoipaper}).

The purpose of this paper is threefold:
\begin{itemize} 
\item Show that Theorem \ref{highfracdim} is a borderline sharp Kakeya 
      result in the sense that (\ref{latticez2}) is the endpoint of a family 
	  of estimates (see (\ref{lattice}) in the statement of Conjecture
	  \ref{family}) such that each member of the family (other than 
	  (\ref{latticez2})) implies a certain sharp Kakeya result that we will 
      formulate in \S 3 below.
\item Show that the sharp Kakeya result is true in certain cases in 
      $\mbb R^3$; see Theorem \ref{hairbrush}.
\item Prove Conjecture \ref{family}	in $\mbb R^2$; see Theorem \ref{mainjj}.  
\end{itemize}

\begin{conj}[when $\beta=2/n$ or $n=2$, this is a theorem]
\label{family}
Suppose $n$, $\alpha$, $R$, $X$, and $\gamma$ are as in the statement of 
Theorem \ref{highfracdim}. 

Let $\beta$ be a parameter satisfying $1/n \leq \beta \leq 2/n$, and define 
the exponent $p$ by 
\begin{displaymath}
p = 2 + \frac{n-\alpha}{n-1} \Big( \frac{2}{n} - \beta \Big).
\end{displaymath}
Then to every $\epsilon > 0$ there is a constant $C_\epsilon$ such that
\begin{equation}
\label{lattice}
\int_X |Ef(x)|^p dx \leq C_\epsilon R^\epsilon \gamma^\beta 
                         R^{\alpha/n} \| f \|_{L^p(\mbb B^{n-1})}^p
\end{equation}
for all $f \in L^p(\mbb B^{n-1})$.
\end{conj}

We note that when $\beta = 2/n$, (\ref{lattice}) becomes (\ref{latticez2}),
so, to prove Conjecture \ref{family} we need to perform the following trade: 
lower the power of $\gamma$ in (\ref{latticez2}) from $2/n$ to $\beta$ in 
return for raising the Lebesgue space exponent from 2 to $p$.

We will show below that if (\ref{lattice}) holds for any $\beta < 2/n$, then
we obtain the sharp Kakeya result of \S 3. 

As noted above, in dimension $n=2$, (\ref{lattice}) is true for all 
$1/2 \leq \beta \leq 1$ (and hence Conjecture \ref{family} is a theorem in 
the plane). We will prove this in the last three sections of the paper by 
using weighted bilinear restriction estimates and the broad-narrow strategy
of \cite{bg:bgmethod}.

Before we discuss the implications of Conjecture \ref{family} to the Kakeya
problem, it will be convenient to write (\ref{lattice}) in an equivalent 
form, which is, perhaps, more user-friendly. This is the purpose of the next 
section.

\section{Writing (\ref{lattice}) in an equivalent form}

Suppose $n \geq 1$ and $0 < \alpha \leq n$. Following \cite{plms12046} (see 
also \cite{dgowwz:falconer} and \cite{pems201030}), for Lebesgue measurable 
functions $H: \mbb R^n \to [0,1]$, we define
\begin{displaymath}
A_\alpha(H)= \inf \Big\{ C : \int_{B(x_0,R)} H(x) dx \leq C R^\alpha
\mbox{ for all } x_0 \in \mbb R^n \mbox{ and } R \geq 1 \Big\},
\end{displaymath}
where $B(x_0,R)$ denotes the ball in $\mbb R^n$ of center $x_0$ and radius 
$R$. We say $H$ is a {\it weight of fractal dimension} $\alpha$ if 
$A_\alpha(H) < \infty$. We note that $A_\beta(H) \leq A_\alpha(H)$ if 
$\beta \geq \alpha$, so we are not really assigning a dimension to the 
function $H$; the phrase ``$H$ is a weight of dimension $\alpha$'' is 
merely another way for us to say that $A_\alpha(H) < \infty$. 

\begin{prop}
\label{equivalent}
Suppose $n$, $\alpha$, $R$, $X$, $\gamma$, $\beta$, and $p$ are as in the 
statement of Conjecture \ref{family}. Then the estimate 
{\rm (\ref{lattice})} holds if and only if to every $\epsilon > 0$ there is 
a constant $C_\epsilon$ such that 
\begin{equation}
\label{weight}
\int_{B(0,R)} |Ef(x)|^p H(x) dx 
\leq C_\epsilon R^\epsilon A_\alpha(H)^\beta R^{\alpha/n} 
     \| f \|_{L^p(\mbb B^{n-1})}^p
\end{equation}
for all functions $f \in L^p(\mbb B^{n-1})$ and weights $H$ of fractal 
dimension $\alpha$.
\end{prop} 

\begin{proof}
Let $H$ be the characteristic function of $X$. By the definition of 
$\gamma$, we have
\begin{displaymath}
\int_{\widetilde{B}(x_0,r)} H(x) dx \leq \gamma \, (r+2)^\alpha
                                    \leq \gamma \, (3r)^\alpha
\end{displaymath}
for all $x_0 \in \mbb R^n$ and $r \geq 1$. Thus $H$ is a weight on 
$\mbb R^n$ of fractal dimension $\alpha$, and 
$A_\alpha(H) \leq 3^\alpha \gamma$. This immediately shows that 
(\ref{weight}) implies (\ref{lattice}).

To prove the reverse implication, we follow
\cite[Proof of Theorem 2.2]{dz:schrodinger3}.

We consider a covering $\{ \widetilde{B} \}$ of $B(0,R)$ by unit lattice 
cubes. Since every unit cube is contained in a ball of radius $\sqrt{n}$, we 
have $\int_{\widetilde{B}} H(x) dx \leq A_\alpha(H) n^{\alpha/2}$, so, if we 
define $v(\widetilde{B})=A_\alpha(H)^{-1} \int_{\widetilde{B}} H(x) dx$ and 
$V_k = \{ \widetilde{B} : 2^{k-1} < n^{-\alpha/2} v(\widetilde{B}) \leq 
2^k \}$, then
\begin{displaymath}
B(0,R) \subset \cup \, \widetilde{B} \subset \cup_{k=-\infty}^0 V_k.
\end{displaymath}

We note that
\begin{eqnarray}
\label{later}
\lefteqn{\int_{\widetilde{B}} H(x) dx \leq 
\Big( \int_{\widetilde{B}} H(x)^{1/\beta} dx \Big)^\beta \leq
\Big( \int_{\widetilde{B}} H(x) dx \Big)^\beta} \nonumber \\
& & = \Big( A_\alpha(H) v(\widetilde{B}) \Big)^\beta \leq 
    n^{\alpha \beta/2} A_\alpha(H)^\beta 2^{k \beta}
\end{eqnarray}
for all $\widetilde{B} \in V_k$, where we have used the assumptions 
$\beta \leq 2/n \leq 1$ and $\| H \|_{L^\infty} \leq 1$.

The vast majority of the sets $V_k$ are negligible for us. In fact, letting
$k_1$ be the sup of the set $\{ k \in \mbb Z :2^k \leq R^{-1000n/\beta} \}$,
we see that 
\begin{eqnarray*}
\int_{\cup_{k=-\infty}^{k_1} \cup_{\widetilde{B} \in V_k}} |Ef(x)|^p H(x) dx
& \leq & \| f \|_{L^1(\mbb B^{n-1})}^p \sum_{k=-\infty}^{k_1} 
         \sum_{\widetilde{B} \in V_k} \int_{\widetilde{B}} H(x) dx \\
& \leq & C A_\alpha(H)^\beta \| f \|_{L^1(\mbb B^{n-1})}^p 
         \sum_{k=-\infty}^{k_1} R^n 2^{k \beta} \\
& \leq & C R^{-999n} A_\alpha(H)^\beta \| f \|_{L^1(\mbb B^{n-1})}^p,		
\end{eqnarray*}
where we used (\ref{later}) on the line before the last, and the fact that
$2^{k_1} \leq R^{-1000n/\beta}$ on the last line. Therefore, we only need to 
estimate
\begin{displaymath}
\int_{\cup_{k=k_1+1}^0 \cup_{\widetilde{B} \in V_k}} |Ef(x)|^p H(x) dx
= \sum_{k=k_1+1}^0 \sum_{\widetilde{B} \in V_k} 
  \int_{\widetilde{B}} |Ef(x)|^p H(x) dx.
\end{displaymath}

Letting $k_0 \in \{ k_1+1, k_1+2, \ldots, 0 \}$ be the integer satisfying 
\begin{displaymath}
\sum_{\widetilde{B} \in V_{k_0}} \int_{\widetilde{B}} |Ef(x)|^p H(x) dx =
\max_{k_1+1 \leq k \leq 0} \Big[ \sum_{\widetilde{B} \in V_k} 
\int_{\widetilde{B}}|Ef(x)|^p H(x) dx \Big],
\end{displaymath}
we see that
\begin{eqnarray}
\label{pigeonholing}
\lefteqn{\int_{B(0,R)} |Ef(x)|^p H(x) dx}  \nonumber \\
& \leq & (-k_1) \sum_{\widetilde{B} \in V_{k_0}} 
         \int_{\widetilde{B}} |Ef(x)|^p H(x) dx 
		 + C R^{-999n} A_\alpha(H)^\beta \| f \|_{L^1(\mbb B^{n-1})}^p.
\end{eqnarray}
Since $-k_1 \lct \log(2R)$, it follows that we only need to estimate
\begin{displaymath}
\sum_{\widetilde{B} \in V_{k_0}} \int_{\widetilde{B}} |Ef(x)|^p H(x) dx.
\end{displaymath}

We start by using the uncertainty principle in the following form. Let 
$d\sigma$ be the pushforward of the $(n-1)$-dimensional Lebesgue measure 
under the map $T : \mbb B^{n-1} \to {\mathcal P}$ given by 
$T(\omega)= (\omega, |\omega|^2)$. Since the measure $d\sigma$ is compactly 
supported and $Ef=\widehat{gd\sigma}$, where $g$ is the function on 
${\mathcal P}$ defined by the equation $f= g \circ T$, it follows that there 
is a non-negative rapidly decaying function $\psi$ on $\mbb R^n$ such that
\begin{displaymath}
\sup_{\widetilde{B}} |Ef|^p \lct |Ef|^p \ast \psi(c(\widetilde{B})),
\end{displaymath}
where $c(\widetilde{B})$ is the center of $\widetilde{B}$. Thus
\begin{displaymath}
\int_{\widetilde{B}} |Ef(x)|^p H(x) dx \lct 
\Big( \int_{\widetilde{B}} H(x) dx \Big) |Ef|^p \ast \psi(c(\widetilde{B})).
\end{displaymath}
From (\ref{later}) we know that 
$\int_{\widetilde{B}} H(x) dx \lct A_\alpha(H)^\beta 2^{k_0 \beta}$ for all
$\widetilde{B} \in V_{k_0}$. Also,
\begin{eqnarray*}
|Ef|^p \ast \psi(c(\widetilde{B}))
&   =  & \int_{B(c(\widetilde{B}),R^\epsilon)} |Ef(x)|^p 
         \psi(c(\widetilde{B})-x) dx \\
&      & + \int_{B(c(\widetilde{B}),R^\epsilon)^c} |Ef(x)|^p 
           \psi(c(\widetilde{B})-x) dx \\		 
& \lct & \int_{B(c(\widetilde{B}),R^\epsilon)} |Ef(x)|^p dx 
         + R^{-1000n} \| f \|_{L^1(\mbb B^{n-1})}^p
\end{eqnarray*}
and
\begin{equation}
\label{bdoverlap}
\sum_{\widetilde{B} \in V_{k_0}} \chi_{B(c(\widetilde{B}),R^\epsilon)} 
\lct R^{n\epsilon},
\end{equation}
so
\begin{eqnarray}
\label{localstar}
\lefteqn{\sum_{\widetilde{B} \in V_{k_0}} \int_{\widetilde{B}} |Ef(x)|^p 
         H(x) dx} \nonumber \\
& \lct & R^{n\epsilon} A_\alpha(H)^\beta 2^{k_0 \beta} \int_V |Ef(x)|^p dx 
         + A_\alpha(H)^\beta R^{-999n} \| f \|_{L^1(\mbb B^{n-1})}^p,
\end{eqnarray}
where $V= \cup_{\widetilde{B} \in V_{k_0}} B(c(\widetilde{B}),R^\epsilon)$.

We now let $\{ \widetilde{B}^* \}$ be the set of all the unit lattice cubes 
that intersect $V$, and $X= \cup \, \widetilde{B}^*$. We plan to apply 
(\ref{lattice}) on this set $X$, but we first need to estimate $\gamma$.

Let $B_r$ be a ball in $\mbb R^n$ of radius $r \geq R^\epsilon$ (if 
$1 \leq r \leq R^\epsilon$, then, clearly, $\# \{ \widetilde{B}^* : 
\widetilde{B}^* \subset B_r \} \lct R^{n \epsilon}$), and $V_r$ the subset 
of $V_{k_0}$ that consists of all unit cubes $\widetilde{B}$ such that 
$B(c(\widetilde{B}),2R^\epsilon) \cap B_r \not= \emptyset$. If $B_r$ 
intersects any of the cubes $\widetilde{B}^*$ that make up $X$, then $B_r$
intersect $B(c(\widetilde{B}),2R^\epsilon)$ for some 
$\widetilde{B} \in V_r$. Therefore, $\gamma \lct R^{n \epsilon} \#(V_r)$. 

Our assumption $r \geq R^\epsilon$, tells us that 
\begin{displaymath}
\cup_{\widetilde{B} \in V_r} B(c(\widetilde{B}),2R^\epsilon) \subset B_{5r},
\end{displaymath}
so (using (\ref{bdoverlap}))
\begin{eqnarray*}
\lefteqn{R^{n \epsilon} \int_{B_{5r}} H(x) dx
\gct \sum_{\widetilde{B} \in V_r} \int_{B(c(\widetilde{B}),2R^\epsilon)} 
H(x) dx} \\
& & \geq \sum_{\widetilde{B} \in V_r} \int_{\widetilde{B}} H(x) dx
    = \sum_{\widetilde{B} \in V_r} v(\widetilde{B}) A_\alpha(H)
    \geq \#(V_r) \, n^{\alpha/2} 2^{k_0-1} A_\alpha(H).
\end{eqnarray*}
On the other hand,
\begin{displaymath}
\int_{B_{5r}} H(x) dx \leq A_\alpha(H) (5r)^\alpha,
\end{displaymath}
so $\#(V_r) \lct R^{n \epsilon} 2^{-k_0} r^\alpha$, and so
$\gamma \lct R^{2n\epsilon} 2^{-k_0}$.

Applying (\ref{lattice}), we now obtain
\begin{displaymath}
\int_V |Ef(x)|^p dx \leq \int_X |Ef(x)|^p dx \lct 
R^{5\epsilon} 2^{-k_0 \beta} R^{\alpha/n} \| f \|_{L^p(\mbb B^{n-1})}^p,
\end{displaymath}
which, combined with (\ref{pigeonholing}) and (\ref{localstar}), implies
that
\begin{eqnarray*}
\int_{B(0,R)} |Ef(x)|^p H(x) dx 
& \lct & R^{(n+6)\epsilon} (2^{k_0})^{\beta-\beta} A_\alpha(H)^\beta 
         R^{\alpha/n} \| f \|_{L^p(\mbb B^{n-1})}^p \\
&  =   & R^{(n+6)\epsilon} A_\alpha(H)^\beta R^{\alpha/n} 
         \| f \|_{L^p(\mbb B^{n-1})}^p,
\end{eqnarray*}
which is our desired estimate (\ref{weight}).
\end{proof}

\section{Conjecture \ref{family} implies a sharp Kakeya result}

Let $\Omega$ be a subset of $\mbb R^n$ that obeys the following property: 
there is a number $\alpha$ between 1 and $n$ such that 
\begin{equation}
\label{omegadim}
|\Omega \cap B_R| \leq C R^\alpha
\end{equation} 
for all balls $B_R$ in $\mbb R^n$ of radius $R \geq 1$. (Here, and 
throughout the paper, $|\mbox{set}| =$ Lebesgue measure of the set.)
 
For large $L$, we divide the unit paraboloid ${\mathcal P}$ into finitely 
overlapping caps $\theta_j$ each of radius $L^{-1}$, and we associate with 
each $\theta_j$ a family $\mbb T_j$ of parallel $1 \times L$ tubes that tile 
$\mbb R^n$ and point in the direction normal to $\theta_j$ at its center. We 
let $N$ be the cardinality of the set 
\begin{equation}
\label{defofJ}
J = \{ j : \mbox{there is a tube of $\mbb T_j$ that lies in 
$\Omega \cap B(0,5L)$} \}.
\end{equation}

It is easy to see that the Kakeya conjecture (in its maximal operator form)  
implies the following bound on $N$: to every $\epsilon > 0$ there is a 
constant $C_\epsilon$ such that
\begin{equation}
\label{sharpbd}
N \leq C_\epsilon L^\epsilon L^{\alpha-1}
\end{equation}
for all $L \geq 1$. In fact, \cite[Proposition 2.2]{demeter:dim4} presents a
proof of the fact that the Kakeya conjecture implies (\ref{sharpbd}) in the 
case when $\Omega$ is a neighborhood of an algebraic variety. This proof 
easily extends to general sets $\Omega$ satisfying (\ref{omegadim}). (For 
the connection between neighborhoods of algebraic varieties and the 
condition (\ref{omegadim}), we refer the reader to \cite{w:tubular}.)

We note that (\ref{sharpbd}) implies that if $\Omega \cap B(0,5L)$ contains
at least one tube from each direction (i.e.\ at least one tube from each of 
the $\sim L^{n-1}$ families $\mbb T_j$), then $\alpha = n$. 

In the special case when $\Omega$ is a neighborhood of an algebraic variety, 
this bound on $N$ was proved by Guth \cite{guth:poly} in $\mbb R^3$, 
conjectured by Guth \cite{guth:poly2} to be true in $\mbb R^n$ for all 
$n \geq 3$, and proved by Zahl \cite{jz:severi} in $\mbb R^4$; see also
\cite{gz:polyaxioms}. The conjecture of \cite{guth:poly2} was then settled 
in all dimensions by Katz and Rogers in \cite{katzrogers}.

In this section we prove that if (\ref{weight}) (or equivalently 
(\ref{lattice})) holds for any $1/n \leq \beta < 2/n$, then (\ref{sharpbd}) 
will follow.

We first write the set $J$ as $\{ j_1, j_2, \ldots, j_N \}$, and for each
$1 \leq l \leq N$, we let $T_l$ be a tube from $\mbb T_{j_l}$ that lies in
$\Omega \cap B(0,5L) = \Omega \cap B_{5L}$. Then  
\begin{eqnarray*}
\lefteqn{NL = \sum_{l=1}^N |T_l| 
         = \sum_{l=1}^N \int_{B_{5L} \cap \, \Omega} \chi_{T_l}(x) \, dx 
		 = \int_{B_{5L} \cap \, \Omega} \sum_{l=1}^N \chi_{T_l}(x) \, dx} \\ 
& & \!\!\!\!\!\!\!\! = L^{2(n-1)} \int_{B_{5L} \cap \, \Omega} \sum_{l=1}^N 
      \Big( \frac{1}{L^{n-1}} \, \chi_{T_l}(x) \Big)^2 dx \lct
	  L^{2(n-1)} \int_{B_{5L} \cap \, \Omega} \sum_{l=1}^N |Ef_l(Lx)|^2 dx,
\end{eqnarray*}
where $f_l$ is supported in $\theta_l$ (in fact, $f_l$ is supported in the 
projection of $\theta_l$ into $\mbb B^{n-1}$) and 
$|Ef_l| \geq |\theta_l| \, \chi_{T_l^*}$, where $T_l^*$ is an appropriate 
$L \times L^2$ tube of the same direction as $T_l$. More precisely,
\begin{displaymath}
|Ef_l(L x)  \geq |\theta_l| \, \chi_{T_l^*}(Lx)= |\theta_l| \, \chi_{T_l}(x)
\gct \frac{1}{L^{n-1}} \, \chi_{T_l}(x)
\end{displaymath}
for all $x \in \mbb R^n$. Letting $H = \chi_\Omega$, we arrive at
\begin{displaymath}
NL \lct L^{2(n-1)} \int_{B_{5L}} \sum_{l=1}^N |Ef_l(Lx)|^2 H(x) dx.
\end{displaymath}

Next, we let $\epsilon_l = \pm 1$ be random signs, define the function 
$f: \mbb B^{n-1} \to \mbb C$ by $f = \sum_{l=1}^N \epsilon_l f_l$, and use
Khintchin's inequality to get
\begin{displaymath}
NL 
\lct L^{2(n-1)} {\mathcal E} \Big( \int_{B_{5L}} |Ef(Lx)|^2 H(x) dx \Big),
\end{displaymath}
where ${\mathcal E}$ is the expectation sign. Since $p \geq 2$, we can apply
H\"{o}lder's inequality in the inner integral to get
\begin{eqnarray*}
N L 
& \lct & L^{2(n-1)} \Big( \int_{B_{5L}} H(x) dx \Big)^{1-(2/p)} 
         {\mathcal E} \Big( \int_{B_{5L}} |Ef(Lx)|^p H(x) dx \Big)^{2/p} \\
& \lct & L^{2(n-1)} L^{\alpha(1-(2/p))} 
         {\mathcal E} \Big( \int_{B_{5L}} |Ef(Lx)|^p H(x) dx \Big)^{2/p}.
\end{eqnarray*}
Applying the change of variables $u=Lx$ and defining the weight $H^*$ by
$H^*(u)= H(x)= H(u/L)$, this becomes
\begin{displaymath}
N L \lct L^{2(n-1)} L^{\alpha(1-(2/p))} L^{-2n/p} 
         {\mathcal E} \Big( \int_{B_{5L^2}} |Ef(u)|^p H^*(u) du \Big)^{2/p},
\end{displaymath}
so that
\begin{equation}
\label{readyforrest}
N L^{3-n} \lct L^{(n+\alpha)(1-(2/p))} 
         {\mathcal E} \Big( \int_{B_{5L^2}} |Ef(u)|^p H^*(u) du \Big)^{2/p}.
\end{equation}

We note that
\begin{eqnarray*}
\lefteqn{\int_{B(u_0,R)} H^*(u) du = L^n \int_{B(u_0/L,R/L)} H(x) dx} \\
& & \leq L^n A_\alpha(H) \big( \frac{R}{L} \big)^\alpha 
    = L^{n-\alpha} A_\alpha(H) R^\alpha
\end{eqnarray*}
if $R \geq L$. On the other hand, if $R \leq L$, then
\begin{displaymath}
\int_{B(u_0,R)} H^*(u) du \lct R^n = R^{n-\alpha} R^\alpha 
\leq L^{n-\alpha} R^\alpha.
\end{displaymath}
Therefore,
\begin{displaymath}
A_\alpha(H^*) \lct L^{n-\alpha}.
\end{displaymath}

We are now in a good shape to apply (\ref{weight}), which tells us that
\begin{eqnarray*}
\int_{B_{5L^2}} |Ef(u)|^p H^*(u) du
& \lct & (L^2)^\epsilon A_\alpha(H^*)^\beta (L^2)^{\alpha/n} 
         \| f \|_{L^p(\mbb B^{n-1})}^p \\
& \lct & L^{2 \epsilon} L^{(n-\alpha)\beta} L^{2\alpha/n} 
         \frac{N}{L^{n-1}}.
\end{eqnarray*}
Inserting this back in (\ref{readyforrest}), we get
\begin{eqnarray*}
N L^{3-n} \lct L^{2 \epsilon} L^{(n+\alpha)(1-(2/p))} 
             \Big( L^{(n-\alpha)\beta} L^{2\alpha /n} L^{1-n} N \Big)^{2/p},	 		
\end{eqnarray*}
so that
\begin{displaymath}
N^{1-(2/p)} L^{3-n} \lct L^{2 \epsilon} L^{(n+\alpha)(1-(2/p))} 
\Big( L^{(n-\alpha)(\beta-\frac{2}{n}) + 2 - \frac{2\alpha}{n}} 
      L^{\frac{2\alpha}{n}} \frac{L^{-2}}{L^{n-3}} \Big)^{2/p},
\end{displaymath}
so that
\begin{displaymath}
N^{1-(2/p)} \lct L^{2 \epsilon} L^{(n-3)(1-(2/p))} L^{(n+\alpha)(1-(2/p))} 
                 L^{(n-\alpha)(\beta-\frac{2}{n})(\frac{2}{p})}.
\end{displaymath}
Therefore,
\begin{displaymath}
N \lct L^{O(\epsilon)} L^{n-3} L^{\alpha} L^{n} 
       L^{\frac{(n-\alpha)(\beta-\frac{2}{n})(\frac{2}{p})}
	           {1- \frac{2}{p}}}
= L^{O(\epsilon)} L^{2n-3+\alpha}
  L^{\frac{(n-\alpha)(\beta-\frac{2}{n})}{\frac{p}{2}-1}}.
\end{displaymath}
But
\begin{displaymath}
\frac{(n-\alpha)(\beta-\frac{2}{n})}{\frac{p}{2}-1} 
= (n-\alpha)(\beta-\frac{2}{n})
  \frac{2(n-1)}{(n-\alpha)(\frac{2}{n}-\beta)} = - 2(n-1) = 2 - 2n,	
\end{displaymath}
so
\begin{displaymath}
N \lct L^{O(\epsilon)} L^{2n-3+\alpha+2-2n} 
     = L^{O(\epsilon)} L^{\alpha - 1}.
\end{displaymath}

\section{Proof of (\ref{sharpbd}) in the regime $1 \leq \alpha \leq 2$ in 
$\mbb R^3$}

The fact that the Kakeya conjecture is true in $\mbb R^2$ tells us that
(\ref{sharpbd}) is also true there. In this section, we use Wolff's 
hairbrush argument from \cite{tw:hairbrush}, as adapted by Guth in 
\cite{guth:poly}, to prove the following bound on $N$.

\begin{thm}
\label{hairbrush}
In $\mbb R^3$, we have
\begin{displaymath}
N \lct \left\{ \begin{array}{ll}
               (\log L) L^{\alpha-1} & \mbox{ if $1 \leq \alpha \leq 2$,} \\
			   (\log L) L^{2\alpha-3} & \mbox{ if $2 \leq \alpha \leq 3$.}
			   \end{array} \right.
\end{displaymath}
\end{thm}

\begin{proof}
Let $\Omega$ be a subset of $\mbb R^3$ that obeys (\ref{omegadim}). As we 
did in the previous section, for large $L$, we consider a decomposition 
$\{ \theta_j \}$ of ${\mathcal P}$ into finitely overlapping caps each of 
radius $L^{-1}$, and we associate with each $\theta_j$ a family $\mbb T_j$ 
of parallel $1 \times L$ tubes that tile $\mbb R^3$ and point in the 
direction of the normal vector $v_j$ of ${\mathcal P}$ at the center of 
$\theta_j$. The quantity $N$ that we need to estimate is the cardinality of 
the set $J$ as defined in (\ref{defofJ}).

For each $j \in J$, we let $T_j$ be the member of $\mbb T_j$ that lies in
$\Omega \cap B(0,5L)$, and $S= \{ T_j \}$. Of course, $N = \#(S)$.   

We tile $\Omega \cap B(0,5L)$ by unit lattice cubes $\widetilde{B}$. Then
(\ref{omegadim}) tells us that 
\begin{equation}
\label{classes}
\#(\{ \widetilde{B} \}) \lct L^\alpha.
\end{equation}
Also, each tube $T_j$ intersects $\sim L$ of the cubes $\widetilde{B}$. 

We now define the function $f : \{ \widetilde{B} \} \to \mbb Z$ by
\begin{displaymath}
f(\widetilde{B}) 
= \# \{ T_j \in S : T_j \cap \widetilde{B} \not= \emptyset \}.
\end{displaymath}
Then
\begin{displaymath}
\sum_{\widetilde{B}} f(\widetilde{B}) \sim N L.
\end{displaymath}
So, by Cauchy-Schwarz and (\ref{classes}),
\begin{displaymath}
N L \lct \Big( \sum_{\widetilde{B}} f(\widetilde{B})^2 \Big)^{1/2} 
         \Big( \#(\{ \widetilde{B} \}) \Big)^{1/2}
\lct \Big( \sum_{\widetilde{B}} f(\widetilde{B})^2 \Big)^{1/2} L^{\alpha/2},
\end{displaymath}
and so
\begin{displaymath}
\sum_{\widetilde{B}} f(\widetilde{B})^2 \gct N^2 L^{2 - \alpha},
\end{displaymath}
which means that the set
\begin{displaymath}
\{ (\widetilde{B}, T_i, T_j) : T_i, T_j \in S, \,
T_i \cap \widetilde{B} \not= \emptyset, \mbox{ and } 
T_j \cap \widetilde{B} \not= \emptyset \}
\end{displaymath}
has cardinality $\gct N^2 L^{2 - \alpha}$. Therefore, the set
\begin{displaymath}
X = \{ (\widetilde{B}, T_i, T_j) : T_i, T_j \in S, \, 
T_i \cap \widetilde{B} \not= \emptyset, \, 
T_j \cap \widetilde{B} \not= \emptyset \mbox{ and } i \not= j \}
\end{displaymath}
has cardinality 
\begin{displaymath}
\geq C_1 N^2 L^{2 - \alpha} - \sum_{\widetilde{B}} f(\widetilde{B})
\geq C_1 N^2 L^{2 - \alpha} - C_2 N L.
\end{displaymath}
If $C_1 N^2 L^{2 - \alpha} \leq 5 C_2 N L$, then 
$N \leq (5C_2/C_1) L^{\alpha-1}$ and the theorem will be proved. So, we may
assume that $N \geq C_3 L^{\alpha-1}$ for some large constant $C_3$. 
Therefore, $\#(X) \gct N^2 L^{2 - \alpha}$.

For $l \in \mbb N$, we define $X_l$ to be the subset of $X$ for which
\begin{displaymath}
\frac{2^{l-1}}{L} \leq \mbox{Angle}(v_i,v_j) \leq \frac{2^l}{L}. 
\end{displaymath} 
Since the angle between any two tubes in our set $S$ ranges between $L^{-1}$ 
and $1$, it follows by the pigeonhole principle that 
$\#(X) \lct (\log L) \#(X_{l_0})$ for some $l_0 \in \mbb N$. Denoting 
$2^{l_0} L^{-1}$ by $\theta$, and $X_{l_0}$ by $X'$, we have 
$L^{-1} \leq \theta \leq 1$ and 
$\#(X') \gct N^2 L^{2 - \alpha} (\log L)^{-1}$.

There are $N$ tubes in $S$. By the pigeonhole principle, one of the tubes 
must appear in 
$\gct N^2 L^{2 - \alpha} (\log L)^{-1}/N = N L^{2 - \alpha} (\log L)^{-1}$ 
of the elements of $X'$. We call this tube $T$, and we define
\begin{displaymath}
\mbb H = \{ T_j \in S : (\widetilde{B}, T, T_j) \in X' \}.
\end{displaymath}
Let $v$ be the direction of the tube $T$. Since the angle between $v$ and
$v_j$ is $\sim \theta$, it follows that $|T \cap T_j| \lct \theta^{-1}$. So, 
the set $\{ \widetilde{B} : (\widetilde{B}, T, T_j) \in X' \}$ has cardinality 
$\lct \theta^{-1}$, and so 
\begin{displaymath}
\#(\mbb H) \gct \frac{N L^{2 - \alpha} (\log L)^{-1}}{\theta^{-1}}
= \theta N L^{2 - \alpha} (\log L)^{-1}.
\end{displaymath}

To finish the proof, we need to also have an upper bound on $\#(\mbb H)$. We 
first observe that
\begin{displaymath}
\bigcup_{T_j \in \mbb H} T_j \subset \Omega \cap \mbox{\bf B}, 
\end{displaymath}
where {\bf B} is a box in $\mbb R^3$ of dimensions 
$L \times \theta L \times \theta L$. Since {\bf B} can be covered by 
$\sim L/(\theta L)$ balls of radius $\theta L$, and since $\theta L \geq 1$, 
the dimensionality property (\ref{omegadim}) tells us that
\begin{displaymath}
\Big| \bigcup_{T_j \in \mbb H} T_j \Big| \lct \theta^{-1} (\theta L)^\alpha.
\end{displaymath}
Next, we use the (by now) standard fact that the tubes $T_j$ in $\mbb H$ are 
morally disjoint (see \cite[Lemma 4.9]{guth:poly} for a very nice 
explanation of this idea) to see that
\begin{displaymath}
\Big| \bigcup_{T_j \in \mbb H} T_j \Big| \gct \#(\mbb H) \, |T_j| 
= \#(\mbb H) \, L.
\end{displaymath}
Therefore,
\begin{displaymath}
\#(\mbb H) \lct \theta^{-1} L^{-1} (\theta L)^\alpha 
= (\theta L)^{\alpha-1}.
\end{displaymath}

Comparing the lower and upper bounds we now have on the cardinality of 
$\mbb H$, we conclude that
\begin{displaymath}
\theta N L^{2 - \alpha} (\log L)^{-1} \lct (\theta L)^{\alpha-1}.
\end{displaymath}
Therefore,
\begin{displaymath}
N \lct (\log L) \theta^{\alpha - 2} L^{2 \alpha - 3}.
\end{displaymath}

If $\alpha \geq 2$, then the fact that $\theta \leq 1$ tells us that
\begin{displaymath}
N \leq (\log L) L^{2 \alpha - 3}.
\end{displaymath}

If $1 \leq \alpha < 2$, then the fact that $\theta \geq 1/L$ tells us that
\begin{displaymath}
N \lct (\log L) (L)^{2 - \alpha} L^{2 \alpha - 3} = (\log L) L^{\alpha - 1}.
\end{displaymath}

It might be interesting for the reader to observe that the sharp result that
we get in the case $1 \leq \alpha < 2$ is due to the fact that we are
using `substantial' information about $\theta$ (namely, $\theta \geq 1/L$),
whereas in the $2 \leq \alpha \leq 3$ we only can use the relatively 
`unsubstantial' information that $\theta \leq 1$.
\end{proof}

\section{Proof of Conjecture \ref{family} in the plane}

The rest of the paper is concerned in proving that Conjecture 2.1 is true in 
$\mbb R^2$. In view of Proposition \ref{equivalent}, this task will be 
accomplished as soon as we prove Theorem \ref{mainjj} below.

We alert the reader that the extension operator in Theorem \ref{mainjj} is 
the one associated with the unit circle $\mbb S^1 \subset \mbb R^2$ and is
given by
\begin{displaymath}
Ef(x) = \int e^{-2 \pi i x \cdot \xi} f(\xi) d\sigma(\xi)
\end{displaymath}
for $f \in L^1(\sigma)$, where $\sigma$ is induced Lebesgue measure on 
$\mbb S^1$. The proof for the extension operator associated with the unit
parabola is similar (and a little easier).

\begin{thm}
\label{mainjj}
Suppose $1 \leq \alpha \leq 2$ and $R \geq 1$. Let $\beta$ be a parameter 
satisfying $1/2 \leq \beta \leq 1$, and define the exponent $p$ by 
\begin{displaymath}
p = 2 + (2 - \alpha)(1 - \beta).
\end{displaymath}
Then to every $\epsilon > 0$ there is a constant $C_\epsilon$ such that
\begin{equation}
\label{mainest}
\int_{B(0,R)} |Ef(x)|^p H(x) dx \leq C_\epsilon R^\epsilon 
A_\alpha(H)^\beta R^{\alpha/2} \| f \|_{L^p(\sigma)}^p
\end{equation}
for all functions $f \in L^p(\sigma)$ and weights $H$ of fractal dimension 
$\alpha$.
\end{thm}

The proof of Theorem \ref{mainjj} will use ideas from \cite{tw:csoipaper},
\cite{mbe:fractal}, \cite{plms12046}, and \cite{dz:schrodinger3}. The 
overarching idea, however, is the broad-narrow strategy of 
\cite{bg:bgmethod}. Implementing this strategy involves
\begin{itemize} 
\item proving a bilinear estimate (see (\ref{bilinearest}) in Subsection 7.1 
      below) that will be used to control $Ef$ on the broad set
\item proving a linear estimate (see (\ref{baselinear}) in Subsection 7.2 
      below) that will be used to establish (\ref{mainest}) when the 
	  function $f$ is supported on an arc of small size (i.e.\ 
	  $\sigma$-measure), which will provide the base of an induction 
	  argument
\item carrying out an induction on the size of the function's support 
      argument that will establish (\ref{mainest}) for general $f$. 	  
\end{itemize}

The main new idea in the proof of Theorem \ref{mainjj} is a localization of 
the weight argument that will help us in deriving the bilinear estimate
(\ref{bilinearest}). We use this argument to take advantage of the locally
constant property of the Fourier transform, and we will end this section by
describing the intuition that lies behind it.

Suppose $R > K^2 \geq 1$, $Q$ is a box in $\mbb R^2$ of dimensions 
$R/K \times R$ (boxes of such dimensions are a common feature in this 
context; see \cite[Subsection 3.2]{dz:schrodinger3} and Subsection 6.2 
below), $f$ is a non-negative function supported in $Q$ that is essentially 
constant at scale $K$, and we are seeking an estimate of the form
\begin{displaymath}
\int_Q f(x) H(x) dx 
\lct K^{-m} A_\alpha(H)^\beta R^{\alpha/2} \| f \|_{L^2(Q)}
\end{displaymath}
for some $m \geq 0$, where $\beta$ is as in the statement of Theorem 
\ref{mainjj}. 

We tile $\mbb R^2$ by cubes $\widetilde{B}_l$ of center $c_l$ and 
side-length $K$, and write
\begin{eqnarray*}
\lefteqn{\int_Q f(x) H(x) dx = \sum_l \int_{\widetilde{B}_l} f(x) H(x) dx 
         \sim \sum_l f(c_l) \int_{\widetilde{B}_l} H(x) dx} \\ 
& & = \sum_l K^{-2} \int_{\widetilde{B}_l} f(c_l) H'(y) dy
    \sim K^{-2} \int_Q f(y) H'(y) dy
\end{eqnarray*}
(recall that $f$ is zero outside $Q$), where $H': \mbb R^2 \to [0,\infty)$ 
is given by
\begin{displaymath}
H'(y) = \int_{\widetilde{B}_l} H(x) dx 
\hspace{0.25in} \mbox{ for } \hspace{0.25in} y \in \widetilde{B}_l.
\end{displaymath}
For $y \in \widetilde{B}_l$, we have
\begin{eqnarray*}
H'(y) 
&  =   & \Big( \int_{\widetilde{B}_l} H(x) dx \Big)^{1-\theta} 
         \Big( \int_{\widetilde{B}_l} H(x) dx \Big)^\theta \\
& \leq & K^{2(1-\theta)} A_\alpha(H)^\theta (\sqrt{2} K)^{\alpha \theta},
\end{eqnarray*}
where $0 \leq \theta \leq 1$ is a parameter that will be determined later in
the argument.

Next, we define the function ${\mathcal H} : \mbb R^2 \to [0,1]$ by
\begin{displaymath}
{\mathcal H}(y) = 2^{-\alpha \theta/2} A_\alpha(H)^{-\theta} 
                  K^{-2(1-\theta)-\alpha \theta} H'(y) 
\end{displaymath}
and observe that
\begin{eqnarray*}
\lefteqn{\int_{B(x_0,r)} {\mathcal H}(y) dy \leq K^2 A_\alpha(H)^{-\theta} 
         K^{-2(1-\theta)-\alpha \theta} \int_{B(x_0,2r)} H'(y) dy} \\
& & \leq K^2 A_\alpha(H)^{-\theta} K^{-2(1-\theta)-\alpha \theta} 
         A_\alpha(H) r^\alpha
    = A_\alpha(H)^{1-\theta} K^{\theta(2-\alpha)} r^\alpha
\end{eqnarray*}
for all $x_0 \in \mbb R^2$ and $R \geq 1$, which means that ${\mathcal H}$ 
is a weight on $\mbb R^2$ of fractal dimension $\alpha$ with
\begin{displaymath}
A_\alpha({\mathcal H}) 
\lct A_\alpha(H)^{1-\theta} K^{\theta(2-\alpha)}.
\end{displaymath} 

Going back to our integral, we now have
\begin{displaymath}
\int_Q f(x) H(x) dx 
\sim A_\alpha(H)^\theta K^{\theta(\alpha-2)} 
     \int_Q f(y) {\mathcal H}(y) dy.
\end{displaymath}
Bounding the integral on the right-hand side by Cauchy-Schwarz, this becomes
\begin{displaymath}
\int_Q f(x) H(x) dx \lct 
A_\alpha(H)^\theta K^{\theta(\alpha-2)} 
\Big( \int_Q {\mathcal H}(y) dy \Big)^{1/2} \| f \|_{L^2(Q)}.
\end{displaymath}
But $Q$ can be covered by $\sim K$ balls of radius $R/K$, so
\begin{eqnarray}
\label{startohold}
\int_Q {\mathcal H}(y) dy 
& \lct & K A_\alpha({\mathcal H}) (K^{-1} R)^\alpha \\
& \lct & A_\alpha(H)^{1-\theta} K^{\theta(2-\alpha)} (K^{-1} R)^\alpha,
         \nonumber
\end{eqnarray}
and so
\begin{displaymath}
\int_{B(0,R)} f(x) H(x) dx \lct 
A_\alpha(H)^{(1+\theta)/2} K^{\theta(\alpha-2)/2} (K^{-1} R)^{\alpha/2} 
\| f \|_{L^2(B(0,R)}.
\end{displaymath}
We now determine $\theta$ by solving the equation $(1+\theta)/2=\beta$, 
which gives $\theta=2\beta-1$, and we arrive at
\begin{displaymath}
\int_{B(0,R)} f(x) H(x) dx 
\lct K^{-m} A_\alpha(H)^\beta R^{\alpha/2} \| f \|_{L^2(B(0,R)}
\end{displaymath}
with $m=\beta+(1-\beta)(\alpha-1)$.

\section{Preliminaries for the proof of Theorem \ref{mainjj}}

This section contains basic facts that we need to prove Theorem \ref{mainjj} 
that we include to make the paper as self-contained as possible.

\subsection{The $L^1$ norm of a rapidly decaying function over a box}

In the rigorous version of the localization argument that we described in 
the previous section, instead of integrating over a proper $R/K \times R$ 
box, we will be integrating against a Schwartz function that is essentially
supported on such a box. It is easy to see that (\ref{startohold}) continues 
to be true in this case. Here are the details. 

Suppose $R_1, \ldots, R_n > 0$ and $\Psi$ is a non-negative Schwartz 
function. For $l= 0, 1, 2, \ldots$, we let $\chi_l$ be the characteristic 
function of the box in $\mbb R^n$ of center $0$ and dimensions 
$2^{l+1} R_1 \times \ldots \times 2^{l+1} R_n$, and $B_l=B(0,2^l)$. Then
\begin{eqnarray*}
\lefteqn{\Psi \Big( \frac{x_1-\nu_1}{R_1}, \ldots, 
         \frac{x_n-\nu_n}{R_n} \Big)} \\
& \leq & \Big( \sup_{B_0} \Psi \Big)
         \chi_{B_0} \Big( \frac{x_1-\nu_1}{R_1}, \ldots, 
		 \frac{x_n-\nu_n}{R_n} \Big) \\
&      & + \sum_{l=1}^\infty 
		   \Big( \sup_{B_l \setminus B_{l-1}} \Psi \Big)
           \chi_{B_l \setminus B_{l-1}} 
		   \Big( \frac{x_1-\nu_1}{R_1}, \ldots, \frac{x_n-\nu_n}{R_n} 
		   \Big) \\
& \lct & \sum_{l=0}^\infty 2^{-Nl} \chi_l(x-\nu)
\end{eqnarray*}
for all $x, \nu \in \mbb R^n$ and $N \in \mbb N$, so that
\begin{displaymath}
\int \Psi \Big( \frac{x_1-\nu_1}{R_1}, \ldots, 
         \frac{x_n-\nu_n}{R_n} \Big) H(x) dx 
\lct \sum_{l=0}^\infty 2^{-Nl} \int_{P_l} H(x) dx, 
\end{displaymath}
where $P_l$ is the box in $\mbb R^n$ of center $\nu$ and dimensions 
$2^{l+1} R_1 \times \ldots \times 2^{l+1} R_n$.

In the special case $R_1= \ldots = R_{n-1} = R/K$ and $R_n=R$ with 
$R \geq K^2 \geq 1$ (as in (\ref{startohold})), this gives
\begin{equation}
\label{contohold}
\int \Psi \Big( \frac{x_1-\nu_1}{R K^{-1}}, \ldots, 
                \frac{x_{n-1}-\nu_{n-1}}{R K^{-1}}
                \frac{x_n-\nu_n}{R} \Big) H(x) dx 
\lct K A_\alpha(H) (K^{-1} R)^\alpha
\end{equation}
for all weights $H$ on $\mbb R^n$ of fractal dimension $\alpha$.

\subsection{A property of $R/K \times \cdots \times R/K \times R$ boxes}

Suppose $R \geq K^2 \geq 1$, $Q$ is an 
$R/K \times \cdots \times R/K \times R$ box in $\mbb R^n$, and $Q^*$ is one
of the dual boxes of $Q$ that is tangent to the unit sphere 
$\mbb S^{n-1} \subset \mbb R^n$. 

Let $\delta = K^{-1}$. Then $Q^*$ has dimensions
$(R \delta)^{-1} \times \ldots \times (R \delta)^{-1} \times R^{-1}$ and its
$(R \delta)^{-1} \times \ldots \times (R \delta)^{-1}$ is tangent to 
$\mbb S^{n-1}$ at some point $e$. It is to be proved that $Q^*$ lies in the
$R^{-1}$-neighborhood of $\mbb S^{n-1}$.

Without any loss of generality, we may assume that $e=(0, \ldots, 0, 1)$.

Suppose $y \in Q^*$. Then
\begin{displaymath}
|y|^2= y_1^2 + \ldots + y_{n-1}^2 + (y_n-1+1)^2
= y_1^2 + \ldots + y_{n-1}^2 + (y_n-1)^2 + 2(y_n-1) +1
\end{displaymath}
so that
\begin{displaymath}
\big| |y|^2-1 \big| \leq y_1^2 + \ldots + y_{n-1}^2 + |y_n-1|^2 + 2 |y_n-1|
\end{displaymath}
so that
\begin{displaymath}
\big| |y|-1 \big| \, \big| |y|+1 \big|
\leq y_1^2 + \ldots + y_{n-1}^2 + 3 |y_n-1|
\end{displaymath}
so that
\begin{displaymath}
\big| |y|-1 \big| \leq y_1^2 + \ldots + y_{n-1}^2 + 3 |y_n-1|
\leq \frac{n-1}{(R \delta)^2} + \frac{3}{R} \lct \frac{1}{R},
\end{displaymath}
where we have used the fact that
\begin{displaymath}
\frac{1}{(R \delta)^2} = \frac{1}{R} \frac{K^2}{R} \leq \frac{1}{R}.
\end{displaymath}

\subsection{The Kakeya information underlying the bilinear estimate}

Suppose $\delta > 0$, $R \geq \delta^{-1}$, and $J_1$ and $J_2$ are subsets
of the circular arc $\{ e^{i\theta} : \pi/4 \leq \theta \leq 3\pi/4 \}$ such 
that $\mbox{Dist}(J_1,J_2) \geq 3\delta$.

Let $N_1$ and $N_2$ be the $R^{-1}$-neighborhoods of $J_1$ and $J_2$,
respectively. In this subsection, we derive a bound on the Lebesgue measure
of the set $(x+N_1) \cap N_2$ for $x \in \mbb R^2$.

Since we are interested in the $L^\infty$-norm of the function
\begin{displaymath}
x \longmapsto \int \chi_{x+N_1}(y) \chi_{N_2}(y) dy,
\end{displaymath}
we let $h \in L^1(\mbb R^2)$ be a non-negative function and consider the 
integral
\begin{displaymath}
I= \int \int \chi_{x+N_1}(y) \chi_{N_2}(y) dy h(x) dx.
\end{displaymath}

Writing
\begin{displaymath}
I = \int \int \chi_{N_1}(y-x) \chi_{N_2}(y) h(x) dy dx
  = \int \chi_{N_2}(y) \int \chi_{N_1}(y-x) h(x) dx dy,
\end{displaymath}
and applying the change of variables $u=y-x$ in the inner integral, we see 
that
\begin{displaymath}
I= \int \chi_{N_2}(y) \int \chi_{N_1}(u) h(y-u) du dy
= \int_{N_2} \int_{N_1} h(y-u) du dy.
\end{displaymath}
Changing into polar coordinates, this becomes
\begin{displaymath}
I= \int_{1-R^{-1}}^{1+R^{-1}} \int_{1-R^{-1}}^{1+R^{-1}} \int_{\tilde{J_1}}
\int_{\tilde{J_2}} h(re^{i\theta}-se^{i\varphi}) r s d\theta d\varphi dr ds,
\end{displaymath}
where $\tilde{J_1}=N_1 \cap \mbb S^1$ and $\tilde{J_2}=N_2 \cap \mbb S^1$.

We define
\begin{displaymath}
T(\theta,\varphi) = re^{i\theta}-se^{i\varphi}
= (r \cos \theta - s \cos \varphi, r \sin \theta - s \sin \varphi).
\end{displaymath}
The Jacobian of this transformation is
\begin{displaymath}
J_T(\theta,\varphi) =
\left| \begin{array}{cc}
       -r \sin \theta & s \sin \varphi \\
       r \cos \theta  & -s \cos \varphi
       \end{array} \right| = r s \sin(\theta-\varphi).
\end{displaymath}
So
\begin{displaymath}
\int_{\tilde{J_1}} \int_{\tilde{J_2}} r s h(re^{i\theta}-se^{i\varphi})
d\theta d\varphi
= \int_{\tilde{J_1} \times \tilde{J_2}}
\frac{h(T(\theta,\varphi)) |J_T|}{|\sin(\theta-\varphi)|} d(\theta,\varphi).
\end{displaymath}
But $|\theta-\varphi| \leq \pi/2$, so
\begin{displaymath}
|\sin(\theta-\varphi)| \geq \frac{2}{\pi} |\theta-\varphi|
\geq \frac{2}{\pi} \mbox{ Dist}(\tilde{J_1},\tilde{J_2})
\geq \frac{2\delta}{\pi},
\end{displaymath}
and so
\begin{eqnarray*}
\int_{\tilde{J_1}} \int_{\tilde{J_2}} r s h(re^{i\theta}-se^{i\varphi})
d\theta d\varphi
& \leq & \frac{\pi}{2\delta} \int_{\tilde{J_1} \times \tilde{J_2}}
        h \circ T(\theta,\varphi) |J_T(\theta,\varphi)| d(\theta,\varphi) \\
&  =   & \frac{\pi}{2\delta} \int_X h(x,y) d(x,y) 
         \; \leq \; \frac{\pi}{2\delta} \| h \|_{L^1}.
\end{eqnarray*}
Thus
\begin{displaymath}
I \leq \int_{1-R^{-1}}^{1+R^{-1}} \int_{1-R^{-1}}^{1+R^{-1}}
\frac{\pi}{2\delta} \| h \|_{L^1} dr ds
= \frac{\pi}{2\delta R^2} \| h \|_{L^1}.
\end{displaymath}
Therefore, by duality,
\begin{equation}
\label{specialD2}
|(x+N_1) \cap N_2| \leq \frac{\pi}{2 R^2 \delta}
\end{equation}
for a.e.\ $x \in \mbb R^2$.

\subsection{Calculation giving the right exponent for the restriction 
            estimate}

(The reader is advised to skip this section until we refer back to it in 
Subsections 7.1 and 7.2.)
 
Suppose $0 < \delta \leq 1$, $1 \leq \alpha \leq n$, 
$1/n \leq \beta \leq 2/n$, $\sigma$ is induced Lebesgue measure on the unit
sphere $\mbb S^{n-1} \subset \mbb R^n$, and $f, g \in L^1(\sigma)$ are 
functions satisfying 
$\sigma(\mbox{supp} \, f), \sigma(\mbox{supp} \, g) \leq \delta^{n-1}$. We 
are looking for an exponent $p \geq 2$ so that
\begin{equation}
\label{rightexplinear}
\| f \|_{L^1(\sigma)}^{p-2} \| f \|_{L^2(\sigma)}^2 
\leq \delta^{(n-\alpha)((2/n)-\beta)} \| f \|_{L^p(\sigma)}^p
\end{equation}
and
\begin{equation}
\label{rightexpbilinear}
\| f \|_{L^1(\sigma)}^{(p/2)-1} \| f \|_{L^2(\sigma)} 
\| g \|_{L^1(\sigma)}^{(p/2)-1} \| g \|_{L^2(\sigma)}
\leq \delta^{(n-\alpha)((2/n)-\beta)}
     \| f \|_{L^p(\sigma)}^{p/2} \| g \|_{L^2(\sigma)}^{p/2}.
\end{equation}

We have
\begin{displaymath}
\| f \|_{L^1(\sigma)} 
\leq \sigma(\mbox{supp} \, f)^{1-(1/p)} \| f \|_{L^p(\sigma)}
\leq \delta^{(n-1)(p-1)/p} \| f \|_{L^p(\sigma)}
\end{displaymath}
and
\begin{displaymath}
\| f \|_{L^2(\sigma)}^2 \leq \sigma(\mbox{supp} \, f)^{1-(2/p)} 
\Big( \int |f|^{2(p/2)} d\sigma \Big)^{2/p}
\leq \delta^{(n-1)(p-2)/p} \| f \|_{L^p(\sigma)}^2,	
\end{displaymath}
so
\begin{eqnarray*}
\| f \|_{L^1(\sigma)}^{p-2} \| f \|_{L^2(\sigma)}^2 
& \leq & \delta^{(n-1)(p-2)(p-1)/p)} \| f \|_{L^p(\sigma)}^{p-2}
         \delta^{(n-1)(p-2)/p} \| f \|_{L^p(\sigma)}^2 \\
&  =   & \delta^{(n-1)(p-2)} \| f \|_{L^p(\sigma)}^p,	
\end{eqnarray*}
so $(n-1)(p-2) = (n-\alpha)((2/n)-\beta)$, and so 
\begin{displaymath}
p = 2 + \frac{n-\alpha}{n-1} \Big( \frac{2}{n} - \beta \Big).
\end{displaymath}

Therefore, (\ref{rightexplinear}) holds with the above value of $p$. Using 
(\ref{rightexplinear}), we now have
\begin{eqnarray*}
\lefteqn{\| f \|_{L^1(\sigma)}^{(p-2)/2} \| f \|_{L^2(\sigma)} 
         \| g \|_{L^1(\sigma)}^{(p-2)/2} \| g \|_{L^2(\sigma)}} \\
& \leq & \Big( \delta^{(n-\alpha)((2/n)-\beta)} \| f \|_{L^p(\sigma)}^p 
         \Big)^{1/2}
         \Big( \delta^{(n-\alpha)((2/n)-\beta)} \| g \|_{L^p(\sigma)}^p 
		 \Big)^{1/2},
\end{eqnarray*}
which is the inequality in (\ref{rightexpbilinear}).

\section{Proof of Theorem \ref{mainjj}}

As the paragraph following the statement of Theorem \ref{mainjj} says, our
proof of this theorem relies on ideas from \cite{tw:csoipaper}, 
\cite{mbe:fractal}, \cite{bg:bgmethod}, \cite{plms12046}, and 
\cite{dz:schrodinger3}. 

\subsection{The bilinear estimate}

Following \cite[pages 1281--1283]{bg:bgmethod}, we write the ball $B(0,R)$ 
as a disjoint union of two sets, one {\it broad}, the other {\it narrow} 
(see Subsection 7.3 below for the definition of these two sets). To estimate 
the $L^p(H dx)$-norm of $Ef$ on the broad set, we consider a bilinear 
estimate.

For the rest of the paper, we will use the following notation. If $\phi$ is
a function on $\mbb R^2$ and $\rho > 0$, then $\phi_\rho$ is the function 
given by $\phi_\rho(\cdot) = \rho^{-2} \phi(\rho^{-1} \cdot)$.

Suppose $f$ is supported in an arc $I$ and $g$ is supported in an arc $J$ 
with $\sigma(I) \sim \sigma(J) \sim \delta$ and
$\delta \leq \mbox{ Dist}(I,J) \leq R^\epsilon \delta$. In this subsection, 
we shall always assume that 
\begin{equation}
\label{condondelta}
(10) R^\epsilon \leq \frac{1}{\delta} \leq \frac{R \delta}{10}.
\end{equation}

Let $\eta$ be a $C_0^\infty$ function on $\mbb R^2$ satisfying
$|\widehat{\eta}| \geq 1$ on $B(0,1)$. Then
\begin{eqnarray*}
\lefteqn{\int_{B(0,R)} |Ef(x) Eg(x)| H(x) dx \; = \;
\int_{B(0,R)} |\widehat{fd\sigma}(x) \widehat{gd\sigma}(x)| H(x) dx} \\
& \leq & \int_{B(0,R)} |\widehat{fd\sigma}(x) \widehat{gd\sigma}(x)| \,
         |\widehat{\eta}(x/R)|^2 H(x) dx \\
&  =   & \int_{B(0,R)} |\big( \eta_{R^{-1}} \ast fd\sigma\widehat{\big)}(x)
         \big( \eta_{R^{-1}} \ast gd\sigma \widehat{\big)}(x)| H(x) dx \\
&  =   & \int_{B(0,R)} |\widehat{F}(x) \widehat{G}(x)| H(x) dx,
\end{eqnarray*}
where $F=\eta_{R^{-1}} \ast fd\sigma$ and $G=\beta_{R^{-1}} \ast gd\sigma$.

Applying the Cauchy-Schwarz inequality in the convolution integral with
respect to the measure $|\eta_{R^{-1}}(\xi-\cdot)|d\sigma$, we see that
\begin{eqnarray*}
\| F \|_{L^2}^2
& \leq & \!\! \int \Big( \int |f(\theta)|^2 \, |\eta_{R^{-1}}(\xi-\theta)|
                         d\sigma(\theta) \Big)
                   \Big( \int |\eta_{R^{-1}}(\xi-\theta)|
                         d\sigma(\theta) \Big) d\xi \\
& \lct & \!\! R \int \int |f(\theta)|^2 \,
              |\eta_{R^{-1}}(\xi-\theta)| d\sigma(\theta) d\xi \\
&   =  & \!\! R \int |f(\theta)|^2 \int |\eta_{R^{-1}}(\xi-\theta)| d\xi
              d\sigma(\theta)
         \; = \; R \, \| \eta \|_{L^1} \, \| f \|_{L^2(\sigma)}^2,
\end{eqnarray*}
where in the second inequality we used the fact that
\begin{displaymath}
\int |\eta_{R^{-1}}(\xi-\theta)|d\sigma(\theta) 
\lct R^2 \sigma(B(\xi,R^{-1}) \lct R.
\end{displaymath}
Therefore,
\begin{equation}
\label{normFG}
\| F \|_{L^2} \lct R^{1/2}  \| f \|_{L^2(\sigma)}
\hspace{0.25in} \mbox{ and } \hspace{0.25in}
\| G \|_{L^2} \lct R^{1/2}  \| g \|_{L^2(\sigma)}.
\end{equation}

Since $F$ is supported in the $R^{-1}$-neighborhood of $I$ and $G$ is 
supported in the $R^{-1}$-neighborhood of $J$, we see (via 
(\ref{condondelta})) that $F$ is supported in a ball of radius 
$(\delta/2) + (\delta/10) = (3\delta/5)$ and similarly for $G$. So 
$F \ast G$ is supported in a ball of radius $(6\delta/5)$, say
$B(\xi_0,(6\delta/5))$. Via the locally constant property of the Fourier
transform, this fact tells us that the Fourier transform of $F \ast G$ is
essentially constant at scale $K=\delta^{-1}$, and hence allows us to
implement the localization of the weight argument that we described in 
Section 5 at the intuitive level, and which we now carry out rigorously.

Let $\phi$ be a Schwartz function which is equal to 1 on $B(0,6/5)$. Then 
$\phi_\delta(\xi-\xi_0)=\delta^{-2}$ on $B(\xi_0,\frac{6\delta}{5})$, so 
that
\begin{displaymath}
F \ast G = \delta^2 \phi_\delta(\cdot - \xi_0) \big( F \ast G \big)
\end{displaymath}
and
\begin{displaymath}
\widehat{F}(x) \widehat{G}(x) = \delta^2
\Big( \phi_\delta(\cdot - \xi_0) \big( F \ast G \big) \widehat{\Big)}(x)
= \delta^2
\big( \phi_\delta(\cdot - \xi_0) \widehat{\big)} \ast \widehat{F \ast G}(x).
\end{displaymath}
Since $\big( \phi_\delta(\cdot - \xi_0) \widehat{\big)}(x)=
e^{-2\pi i x \cdot \xi_0} \widehat{\phi}(\delta x)$, it follows that
\begin{eqnarray*}
\widehat{F}(x) \widehat{G}(x)
& = & \delta^2 \int \big( \phi_\delta(\cdot - \xi_0) \widehat{\big)}(x-y)
      \widehat{F \ast G}(y) dy \\
& = & \delta^2 \int e^{-2\pi i (x-y) \cdot \xi_0}
      \widehat{\phi}(\delta(x-y)) \widehat{F \ast G}(y) dy,
\end{eqnarray*}
so that
\begin{displaymath}
|\widehat{F}(x) \widehat{G}(x)| \leq \delta^2
\int |\widehat{\phi}(\delta(x-y))| \, |\widehat{F \ast G}(y)| dy.
\end{displaymath}
Therefore,
\begin{equation}
\label{star1}
\int_{B(0,R)} \!\! |Ef(x) Eg(x)| H(x) dx
\leq \delta^2 \int |\widehat{F \ast G}(y)| \!\!
     \int |\widehat{\phi}(\delta(x-y))| H(x) dx dy.		 
\end{equation}

For $l = 0, 1, 2, \ldots$, we let $B_l=B(y,2^l \delta^{-1})$ and write
\begin{eqnarray*}
\lefteqn{\int |\widehat{\phi}(\delta(x-y))| H(x) dx} \\
& = & \int_{B_0} |\widehat{\phi}(\delta(x-y))| H(x) dx + \sum_{l=1}^\infty
      \int_{B_l \setminus B_{l-1}} |\widehat{\phi}(\delta(x-y))| H(x) dx \\
& \leq & \int_{B_0} \frac{C_N H(x)}{(1+\delta|x-y|)^N} dx
         + \sum_{l=1}^\infty \int_{B_l \setminus B_{l-1}}
         \frac{C_N H(x)}{(1+\delta|x-y|)^N} dx \\
& \leq & C_N \int_{B_0} H(x) dx + \sum_{l=1}^\infty 
        \frac{C_N}{\big( 1+\delta \frac{2^{l-1}}{\delta} \big)^N} 
		\int_{B_l} H(x) dx. 
\end{eqnarray*}
We now let $0 \leq \theta \leq 1$ be a parameter that will be determined 
later and write
\begin{eqnarray*}
\int_{B_l} H(x) dx
& = & \Big( \int_{B_l} H(x) dx \Big)^{1-\theta} 
      \Big( \int_{B_l} H(x) dx \Big)^\theta \\
& \leq & |B_l|^{1-\theta} \Big( A_\alpha(H) 
         \Big( \frac{2^l}{\delta} \Big)^\alpha \Big)^\theta \\
& \leq & C_\theta \Big( \frac{2^l}{\delta} \Big)^{2(1-\theta)+\alpha \theta}
         A_\alpha(H)^\theta,
\end{eqnarray*}
where we have used the fact that $1/\delta \geq 1$, and we obtain
\begin{eqnarray*}
\lefteqn{\int |\widehat{\phi}(\delta(x-y))| H(x) dx} \\
& \leq & C_{N,\theta} \Big( \frac{1}{\delta} 
         \Big)^{n(1-\theta)+\alpha \theta} A_\alpha(H)^\theta
         + \sum_{l=1}^\infty \frac{C_N}{(1+2^{l-1})^N}
		   \Big( \frac{2^l}{\delta} \Big)^{2(1-\theta)+\alpha \theta}
		            A_\alpha(H)^\theta \\
& \leq & C_{N,\theta} A_\alpha(H)^\theta 
         \Big( \frac{1}{\delta} \Big)^{2(1-\theta)+\alpha \theta}.
\end{eqnarray*}
Also,
\begin{displaymath}
\int_{B(x_0,r)} \int |\widehat{\phi}(\delta(x-y))| H(x) dx dy = \int \int
\chi_{B(x_0,r)}(y) |\widehat{\phi}(\delta(x-y))| dy H(x) dx.
\end{displaymath}
Applying the change of variables $z=\delta(x-y)$ in the inner integral, we
get
\begin{eqnarray*}
\int_{B(x_0,r)} \int |\widehat{\phi}(\delta(x-y))| H(x) dx dy \!\!
& = & \!\! \frac{1}{\delta^2} \int \int \chi_{B(x_0,r)} 
      \big( x-\frac{z}{\delta} \big) |\widehat{\phi}(z)| dz H(x) dx \\
& = &  \!\! \frac{1}{\delta^2} \int |\widehat{\phi}(z)|
      \int \chi_{B(x_0,r)} \big( x-\frac{z}{\delta} \big) H(x) dx dz.
\end{eqnarray*}
But
\begin{displaymath}
\int \chi_{B(x_0,r)} \big( x-\frac{z}{\delta} \big) H(x) dx
= \int_{B(x_0+\frac{z}{\delta},r)} H(x) dx \leq A_\alpha(H) r^\alpha
\end{displaymath}
for all $x_0 \in \mbb R^n$ and $r \geq 1$, so
\begin{displaymath}
\int_{B(x_0,r)} \int |\widehat{\phi}(\delta(x-y))| H(x) dx dy \leq
\frac{1}{\delta^2} \, \| \widehat{\phi} \|_{L^1} \, A_\alpha(H) \, r^\alpha
\end{displaymath}
for all $x_0 \in \mbb R^2$ and $r \geq 1$.

For $y \in \mbb R^2$, define
\begin{displaymath}
{\mathcal H}(y)= 
\frac{\delta^{2(1-\theta)+\alpha \theta}}{C_{N,\theta} A_\alpha(H)^\theta}
\int |\widehat{\phi}(\delta(x-y))| H(x) dx.
\end{displaymath}
In view of the above discussion, we have
\begin{displaymath}
\| {\mathcal H} \|_{L^\infty} \leq 1
\hspace{0.5in} \mbox{ and } \hspace{0.5in}
\int_{B(x_0,r)} {\mathcal H}(y) dy 
\leq C A_\alpha^{1-\theta} \delta^{(\alpha-2)\theta} r^\alpha
\end{displaymath}
for all $x_0 \in \mbb R^2$ and $r \geq 1$. Thus ${\mathcal H}$ is a weight
on $\mbb R^2$ of fractal dimension $\alpha$ with
\begin{displaymath}
A_\alpha({\mathcal H}) 
\leq C A_\alpha(H)^{1-\theta} \delta^{(\alpha-2)\theta}.
\end{displaymath}

Going back to (\ref{star1}), we now have
\begin{eqnarray}
\label{goingback}
\int_{B(0,R)} |Ef(x) Eg(x)| H(x) dx 
& \leq & \delta^2 \frac{C_{N,\theta} A_\alpha(H)^\theta}
                       {\delta^{2(1-\theta)+\alpha \theta}} 
	     \int |\widehat{F \ast G}(y)| {\mathcal H}(y) dy \nonumber \\
&  =   & C_{N,\theta} \, \delta^{(2-\alpha)\theta} A_\alpha(H)^\theta
         \int |\widehat{F \ast G}(y)| {\mathcal H}(y) dy.
\end{eqnarray}

Next, we let $Q^*$ be the box in frequency space (where the circle is 
located) of dimensions $(R\delta)^{-1} \times R^{-1}$, centered at the 
origin, and with the $(R\delta)^{-1}$-side (i.e.\ the long side) parallel to
the line segment that connects the midpoint of $I$ to that of $J$. We also 
let $\{ Q_l \}$ be a tiling of $\mbb R^2$ by boxes dual to $Q^*$ (i.e.\ each
$Q_l$ is an $R\delta \times R$ box whose $R\delta$-side is parallel to the
$(R\delta)^{-1}$-side of $Q^*$) with centers $\{ \nu_l \}$, $\psi$ be a 
$C_0^\infty$ function on $\mbb R^2$, and we define
\begin{displaymath}
\psi_l(\xi)= (R\delta) R \; 
             \psi(R\delta \xi_1, R\xi_2) \; e^{2\pi i \nu_l \cdot \xi}.
\end{displaymath}
In the definition of $\psi_l$, we are assuming that the line joining the 
midpoint of $I$ to that of $J$ is horizontal (i.e.\ parallel to the 
$\xi_1$-axis). This assumption makes the presentation a little smoother and,
of course, does not cost us any loss of generality.

We assume further that the Fourier transform of $\psi$ is non-negative and
satisfies $\widehat{\psi} \geq 1/2$ on $[-1/2,1/2] \times [-1/2,1/2]$. Then
\begin{displaymath}
\widehat{\psi_l}(x) = \widehat{\psi} 
\Big( \frac{x_1-\nu_{l,1}}{R\delta}, \frac{x_2-\nu_{l,2}}{R} \Big)
\geq \frac{1}{2}
\hspace{0.25in} \mbox{ if } \hspace{0.25in} x \in Q_l.
\end{displaymath}
By the Schwartz decay of $\widehat{\psi}$, we have
$\sum_{m \in \mbb Z^2} \widehat{\psi}(\cdot - m)^k \lct 1$ for any 
$k \in \mbb N$. Also, $\{ \nu_l \}$ is basically 
$R\delta \mbb Z \times R \mbb Z$, so
\begin{displaymath}
\sum_{l=1}^\infty \widehat{\psi_l}(R\delta x_1, R x_2)^k
= \sum_{l=1}^\infty \widehat{\psi} \Big(
  \frac{R\delta x_1-\nu_{l,1}}{R\delta}, \frac{Rx_2-\nu_{l,2}}{R} \Big)^k
= \sum_{m \in \mbb Z^2} \widehat{\psi}(x - m)^k \lct 1,
\end{displaymath}
and so
\begin{equation}
\label{basicallyzz}
\sum_{l=1}^\infty \widehat{\psi_l}(x)^k \lct 1
\end{equation}
for all $x \in \mbb R^2$.

Going back to (\ref{goingback}), we can now write
\begin{displaymath}
\int_{B(0,R)} |Ef(x) Eg(x)| H(x) dx
\lct \delta^{(2-\alpha)\theta} A_\alpha(H)^\theta \sum_{l=1}^\infty
     \int |\widehat{F}(x) \widehat{G}(x)| \widehat{\psi_l}(x)^3 
	 {\mathcal H}(x) dx.
\end{displaymath}
Letting $F_l=\psi_l \ast F$ and $G_l=\psi_l \ast G$, this becomes
\begin{displaymath}
\int_{B(0,R)} |Ef(x) Eg(x)| H(x) dx
\lct \delta^{(2-\alpha)\theta} A_\alpha(H)^\theta \sum_{l=1}^\infty
     \int |\widehat{F_l}(x) \widehat{G_l}(x)| \widehat{\psi_l}(x) 
	 {\mathcal H}(x) dx.
\end{displaymath}
By Cauchy-Schwarz,
\begin{displaymath}
\int |\widehat{F_l}(x) \widehat{G_l}(x)| \widehat{\psi_l}(x) 
	 {\mathcal H}(x) dx
\leq \| \widehat{F_l} \widehat{G_l} \|_{L^2} 
     \| \widehat{\psi_l}(x) {\mathcal H} \|_{L^2}.
\end{displaymath}
Applying (\ref{contohold}) from Subsection 6.1 with $n=2$ and 
$K = \delta^{-1}$, we have
\begin{eqnarray*}
\lefteqn{\int \widehat{\psi_l}(x)^2 {\mathcal H}(x)^2 dx
\; \lct \; \int \widehat{\psi_l}(x) {\mathcal H}(x) dx
\; \lct \; A_\alpha({\mathcal H}) \frac{R}{R\delta} (R\delta)^\alpha} \\
& & \lct \; A_\alpha(H)^{1-\theta} \delta^{(\alpha-2)\theta} R^\alpha 
         \delta^{\alpha-1}
\; = \; A_\alpha(H)^{1-\theta} \delta^{(\alpha-2)\theta+\alpha-1} R^\alpha,
\end{eqnarray*}
so that
\begin{displaymath}
\| \widehat{\psi_l}(x) {\mathcal H} \|_{L^2} \lct A_\alpha(H)^{(1-\theta)/2} 
\delta^{((\alpha-2)\theta+\alpha-1)/2} R^{\alpha/2}.
\end{displaymath}
Therefore,
\begin{displaymath}
\int_{B(0,R)} |Ef(x) Eg(x)| H(x) dx 
\lct A_\alpha(H)^{(1+\theta)/2} \delta^{((2-\alpha)\theta+\alpha-1)/2} 
     R^{\alpha/2} \sum_{l=1}^\infty \| \widehat{F_l} \widehat{G_l} \|_{L^2}.
\end{displaymath}
Letting $\beta = (1+\theta)/2$ (since $0 \leq \theta \leq 1$, we have
$1/2 \leq \beta \leq 1$), this becomes
\begin{displaymath}
\int_{B(0,R)} |Ef(x) Eg(x)| H(x) dx
\lct A_\alpha(H)^\beta \delta^{(2-\alpha)\beta+\alpha-(3/2)} R^{\alpha/2}
     \sum_{l=1}^\infty \| \widehat{F_l} \widehat{G_l} \|_{L^2}.	
\end{displaymath}

We now let $A_l$ be the support of $F_l$, $B_l$ be the support of $G_l$, and
define the function $\lambda_l : \mbb R^2 \to [0,\infty)$ by
$\lambda_l(\xi)= |(\xi-A_l) \cap B_l|$. Applying Plancherel's theorem 
followed by Cauchy-Schwarz, we see that
\begin{displaymath}
\| \widehat{F_l} \widehat{G_l} \|_{L^2}^2 = \int |F_l \ast G_l(\xi)|^2 d\xi
\leq \| \lambda_l \|_{L^\infty} \int |F_l|^2 \ast |G_l|^2(\xi) d\xi.
\end{displaymath} 
By Young's inequality,
\begin{displaymath}
\int |F_l|^2 \ast |G_l|^2(\xi) d\xi 
\leq \| |F_l|^2 \|_{L^1} \| |G_l|^2 \|_{L^1}
= \| F_l \|_{L^2}^2 \| G_l \|_{L^2}^2,
\end{displaymath}
so the only problem is to estimate $\| \lambda_l \|_{L^\infty}$. We will do
this by using the Kakeya bound (\ref{specialD2}) of Subsection 6.3. 

Our assumptions on the arcs $I$ and $J$ imply that the angle between any two
points in $I \cup J$ is \!\!\!\!\! $\lct R^\epsilon \delta$. Also, for each 
$l$, the function $\psi_l$ is supported in the 
$(R\delta)^{-1} \times R^{-1}$ box $Q^*$ of center $(0,0)$ and with the long 
side parallel to the line joining the midpoints of $I$ and $J$. So, if 
$e \in I \cup J$, then the translate $Q^*+e$ of $Q^*$ is contained in an 
$(R \delta)^{-1} \times R^{\epsilon-1}$ box with the $(R \delta)^{-1}$-side 
tangent to $\mbb S^1$ at $e$. Therefore, the property of boxes of this form 
that was presented in Subsection 6.2 tells us that $Q^*+e$ is contained in 
the $R^{\epsilon-1}$-neighborhood of $\mbb S^1$. Therefore, the sets $A_l$ 
and $B_l$ satisfy the requirements needed for us to apply (\ref{specialD2}) 
and conclude 
\begin{displaymath}
\| \lambda_l \|_{L^\infty} \lct \frac{R^\epsilon}{R^2\delta}.
\end{displaymath}

Putting together what we have proved in the previous two paragraphs, we 
obtain
\begin{displaymath}
\| \widehat{F_l} \widehat{G_l} \|_{L^2}^2 
\lct \frac{R^\epsilon}{R^2\delta} \, \| F_l \|_{L^2}^2 \| G_l \|_{L^2}^2,
\end{displaymath}
and hence
\begin{eqnarray*}
\lefteqn{\int_{B(0,R)} |Ef(x) Eg(x)| H(x) dx} \\
& \lct & R^\epsilon A_\alpha(H)^\beta 
         \delta^{(2-\alpha)\beta+\alpha-(3/2)} 
         \frac{R^{\alpha/2}}{(R^2 \delta)^{1/2}} 
		 \sum_{l=1}^\infty \| F_l \|_{L^2} \| G_l \|_{L^2} \\
&  =  & R^\epsilon A_\alpha(H)^\beta 
         \delta^{(2-\alpha)(\beta-1)} \frac{R^{\alpha/2}}{R} 
		 \sum_{l=1}^\infty \| F_l \|_{L^2} \| G_l \|_{L^2}.
\end{eqnarray*}

By Cauchy-Schwarz and Plancherel,
\begin{displaymath}
\sum_{l=1}^\infty \| F_l \|_{L^2} \| G_l \|_{L^2}
\leq \Big( \sum_{l=1}^\infty \| \widehat{F_l} \|_{L^2}^2 \Big)^{1/2}
     \Big( \sum_{l=1}^\infty \| \widehat{G_l} \|_{L^2}^2 \Big)^{1/2}.
\end{displaymath}
Also, by (\ref{basicallyzz}),
\begin{displaymath}
\sum_{l=1}^\infty \| \widehat{F_l} \|_{L^2}^2
= \int |\widehat{F}(x)|^2 \sum_{l=1}^\infty \widehat{\psi_l}(x)^2 dx
\lct \| \widehat{F} \|_{L^2}^2 =  \| F \|_{L^2}^2
\end{displaymath}
and similarly for $\sum_{l=1}^\infty \| \widehat{G_l} \|_{L^2}^2$, so
\begin{displaymath}
\int_{B(0,R)} |Ef(x) Eg(x)| H(x) dx
\lct R^\epsilon A_\alpha(H)^\beta \delta^{(2-\alpha)(\beta-1)} 
     \frac{R^{\alpha/2}}{R} \| F \|_{L^2} \| G \|_{L^2}.
\end{displaymath}
Recalling (\ref{normFG}), our bilinear estimate becomes
\begin{displaymath}
\int_{B(0,R)} |Ef(x) Eg(x)| H(x) dx
\lct R^\epsilon A_\alpha(H)^\beta \delta^{(2-\alpha)(\beta-1)} R^{\alpha/2} 
     \| f \|_{L^2(\sigma)} \| g \|_{L^2(\sigma)}.
\end{displaymath}

Writing
\begin{eqnarray*}
\lefteqn{\int_{B(0,R)} |Ef(x) Eg(x)|^{p/2} H(x) dx} \\
&   =  & \int_{B(0,R)} |Ef(x) Eg(x)|^{(p/2)-1} |Ef(x) Eg(x)| H(x) dx \\
& \leq & \| f \|_{L^1(S)}^{(p/2)-1} \| g \|_{L^1(S)}^{(p/2)-1} 
         \int_{B(0,R)} |Ef(x) Eg(x)| H(x) dx \\
& \leq & C_B R^\epsilon A_\alpha(H)^\beta R^{\alpha/2} 
         \delta^{(2-\alpha)(\beta-1)} 
         \| f \|_{L^1(\sigma)}^{(p/2)-1} \| f \|_{L^2(\sigma)} 
		 \| g \|_{L^1(\sigma)}^{(p/2)-1} \| g \|_{L^2(\sigma)}		 
\end{eqnarray*}
and applying (\ref{rightexpbilinear}), we arrive at
\begin{equation}
\label{bilinearest}
\int_{B(0,R)} |Ef(x) Eg(x)|^{p/2} H(x) dx
\leq R^\epsilon C_B A_\alpha(H)^\beta R^{\alpha/2} 
         \| f \|_{L^p(\sigma)}^{p/2} \| g \|_{L^p(\sigma)}^{p/2}.		 
\end{equation}

\subsection{The linear estimate}

In this subsection, we work in $\mbb R^n$ with $n \geq 2$.

Suppose $f$ is supported in a cap of radius $\delta/2$. The purpose of this 
section is to establish the estimate
\begin{equation}
\label{baselinear}
\int_{B(0,R)} |Ef(x)|^p H(x) dx \leq C_L A_\alpha(H)^\beta 
\delta^{-2\alpha/n} (\delta^2 R) \| f \|_{L^p(\sigma)}^p
\end{equation}
under the condition
\begin{equation}
\label{condondeltalinear}
(10) R^\epsilon \leq \frac{1}{\delta} \leq \frac{R}{10}.
\end{equation}

Let $\eta$ be a $C_0^\infty$ function on $\mbb R^n$ satisfying 
$|\widehat{\eta}| \geq 1$ on $B(0,1)$, and $F=\eta_{R^{-1}} \ast f d\sigma$. 
Then
\begin{displaymath}
\int_{B(0,R)} |Ef(x)|^2 H(x)dx \leq \int_{B(0,R)} |\widehat{F}(x)|^2 H(x)dx.
\end{displaymath}
Also, let $\psi$ be a $C_0^\infty$ function on $\mbb R^n$, and $\{ B_l \}$ 
be a tiling of $\mbb R^n$ by balls dual to $B(0,\delta)$ (i.e.\ 
$\delta^{-1}$-balls) with centers $\{ \nu_l \}$, and set
\begin{displaymath}
\psi_l(\xi)=\delta^{-n} \psi(\delta^{-1} \xi) \, e^{2\pi i \nu_l \cdot \xi}.
\end{displaymath}
We assume further that $\widehat{\psi}$ is non-negative and $\geq 1/2$ on 
the unit ball. Then
\begin{displaymath}
\widehat{\psi_l}(x)= \widehat{\psi}(\delta (x-\tau_l)) \geq \frac{1}{2}
\end{displaymath}
if $|\delta (x-\tau_l)| \leq 1$, i.e.\ if $x \in B_l$. Thus
\begin{displaymath}
\int_{B(0,R)} |Ef(x)|^2 H(x)dx \lct \sum_{l=1}^\infty 
\int |\widehat{F}(x) \widehat{\psi_l}(x)|^2 \widehat{\psi_l}(x) H(x) dx.
\end{displaymath}
Since $1/n \leq \beta \leq 2/n$, we can apply H\"{o}lder's inequality
with the dual exponents $1/(1-\beta)$ and $1/\beta$ to get
\begin{displaymath}
\int_{B(0,R)} |Ef(x)|^2 H(x)dx \lct \sum_{l=1}^\infty 
\| \widehat{F \ast \psi_l} \|_{L^{2/(1-\beta)}}^2
\| \widehat{\psi_l} H \|_{L^{1/\beta}}.
\end{displaymath}
Since $\| H \|_{L^\infty} \leq 1$, we have
\begin{displaymath}
\| \widehat{\psi_l} H \|_{L^{1/\beta}}^{1/\beta}
\leq \int \widehat{\psi_l}(x)^{1/\beta} H(x) dx,
\end{displaymath}
and hence (by the proof of (\ref{contohold}))
\begin{displaymath}
\| \widehat{\psi_l} H \|_{L^{1/\beta}}^{1/\beta} 
\lct A_\alpha(H) \Big( \frac{1}{\delta} \Big)^\alpha.
\end{displaymath}
Also, by Hausdorff-Young,
\begin{displaymath}
\| \widehat{F \ast \psi_l} \|_{L^{2/(1-\beta)}} 
\leq \| F \ast \psi_l \|_{L^{2/(1+\beta)}}.
\end{displaymath}
Therefore,
\begin{displaymath}
\int_{B(0,R)} |Ef(x)|^2 H(x) dx 
\lct A_\alpha(H)^\beta \delta^{-\alpha \beta} 
\sum_{l=1}^\infty \| F \ast \psi_l \|_{L^{2/(1+\beta)}}^2.
\end{displaymath}

Since (\ref{condondeltalinear}) tells us $1/R \leq \delta/10$, it follows 
that $F$ is supported in a ball of radius 
$(\delta/2)+(\delta/10)= (3/5) \delta$, say $B(\xi_0,3\delta/5)$. Moreover, 
since $\psi_l$ is supported in $B(0,\delta)$, it follows by H\"{o}lder's 
inequality and Plancherel's theorem that
\begin{displaymath}
\| F \ast \psi_l \|_{L^{2/(1+\beta)}}^2 
\lct \delta^{n \beta} \| F \ast \psi_l \|_{L^2}^2
= \delta^{n \beta} \| \widehat{F} \widehat{\psi_l} \|_{L^2}^2.
\end{displaymath}
Thus
\begin{eqnarray*}
\int_{B(0,R)} |Ef(x)|^2 H(x) dx 
& \lct & A_\alpha(H)^\beta \delta^{-\alpha \beta} \delta^{n \beta}
         \sum_{l=1}^\infty 
         \int |\widehat{F}(\xi) \widehat{\psi_l}(\xi)|^2 d\xi \\
&  =   & A_\alpha(H)^\beta \delta^{(n-\alpha)\beta} 
         \int |\widehat{F}(\xi)|^2 
         \sum_{l=1}^\infty |\widehat{\psi_l}(\xi)|^2 d\xi \\
& \lct & A_\alpha(H)^\beta \delta^{(n-\alpha)(\beta-(2/n))} 
         \delta^{2-(2\alpha/n)} \| F \|_{L^2}^2.
\end{eqnarray*}
But we know from (\ref{normFG}) (whose proof shows that it is true in 
$\mbb R^n$ for all $n \geq 2$) that 
$\| F \|_{L^2} \lct \sqrt{R} \, \| f \|_{L^2(\sigma)}$, so
\begin{displaymath}
\int_{B(0,R)} |Ef(x)|^2 H(x) dx
\lct A_\alpha(H)^\beta \delta^{-2 \alpha/n} (\delta^2 R) 
     \delta^{(n-\alpha)(\beta-(2/n))} \| f \|_{L^2(\sigma)}^2.
\end{displaymath}

Writing 
\begin{displaymath}
|Ef(x)|^p= |Ef(x)|^{p-2} |Ef(x)|^2 \leq \| f \|_{L^1(\sigma)}^{p-2} |Ef(x)|^2
\end{displaymath}
and using (\ref{rightexplinear}), we now see that
\begin{eqnarray*}
\lefteqn{\int_{B(0,R)} |Ef(x)|^p H(x) dx} \\
& \lct & A_\alpha(H)^\beta \delta^{-2 \alpha/n} (\delta^2 R)
         \delta^{(n-\alpha)(\beta-(2/n))} \| f \|_{L^1(\sigma)}^{p-2} 
		 \| f \|_{L^2(\sigma)}^2 \\
& \lct & A_\alpha(H)^\beta \delta^{-2 \alpha/n} (\delta^2 R)
         \| f \|_{L^p(\sigma)}^p,
\end{eqnarray*}
which proves (\ref{baselinear}).

\subsection{The induction argument}

We let $0 < \epsilon < 10^{-2}$ and $R \geq 1$ be two numbers satisfying 
$R \geq (1000)^{1/(1-4\epsilon)}$. We also let $\delta$ be as in
(\ref{condondelta}). We're going to prove our estimate by inducting over 
$\delta$.

\underline{Base of the induction:} Here $\delta= R^{-1/2}$. Plugging this
value of $\delta$ into (\ref{baselinear}) in dimension $n=2$, we get
\begin{displaymath}
\int_{B(0,R)} |Ef(x)|^p H(x) dx 
\leq C_L A_\alpha(H)^\beta R^{\alpha /2} \| f \|_{L^p(\sigma)}^p.
\end{displaymath}

\underline{The inductive step:} Suppose $\delta$ satisfies the condition 
(\ref{condondelta}):
\begin{displaymath}
(10) R^\epsilon \leq \frac{1}{\delta} \leq \frac{R \delta}{10},
\end{displaymath}
and the estimate is true for $\delta$, i.e.
\begin{equation}
\label{indhyp}
\int_{B(0,R)} |Ef(x)|^p H(x) dx 
\leq C R^\epsilon A_\alpha(H)^\beta R^{\alpha /2} \| f \|_{L^p(\sigma)}^p.
\end{equation}
whenever $f \in L^1(\sigma)$, $f$ is supported on an arc 
$I_\delta \subset \mbb S^1$, and $\sigma(I_\delta) \leq \delta$. We are 
going to show that 
\begin{equation}
\label{indest}
\int_{B(0,R)} |Eg(x)|^p H(x) dx 
\leq C' R^\epsilon A_\alpha(H)^\beta R^{\alpha /2} \| g \|_{L^p(\sigma)}^p
\end{equation}
whenever $g \in L^1(\sigma)$, $f$ is supported on an arc 
$I_{R^\epsilon \delta} \subset \mbb S^1$, and 
$\sigma(I_{R^\epsilon \delta}) \leq R^\epsilon \delta$, where
\begin{displaymath}
C'= 3^p C + (10)^p R^{(p+2)\epsilon} C_B.
\end{displaymath}

We let $K = R^\epsilon$ and cover the support of $g$ by $K$ arcs $\tau$ each 
of measure $\delta$. We then write $g= \sum_{\tau} f_\tau$ with each 
function $f_\tau$ supported in the arc $\tau$. 

Following \cite{bg:bgmethod} and \cite{guth:poly}, for $x \in \mbb R^2$, we 
define the significant set of $x$ by
\begin{displaymath}
S(x)= \{ \tau : |Ef_\tau(x)| \geq \frac{1}{10 K} |Eg(x)| \}.
\end{displaymath}
Then
\begin{displaymath}
|Eg(x)| 
\leq \Big| \sum_{\tau \in S(x)} Ef_\tau(x) \Big| + \frac{1}{10} |Eg(x)|,
\end{displaymath}
so that
\begin{equation}
\label{boundoneg}
|Eg(x)| \leq \frac{10}{9} \Big| \sum_{\tau \in S(x)} Ef_\tau(x) \Big|.
\end{equation}

The narrow set ${\mathcal N}$ and the broad set ${\mathcal B}$ are now 
defined as
\begin{displaymath}
{\mathcal N} = B(0,R) \cap \{ x \in \mbb R^2 : \# S(x) \leq 2 \} 
\hspace{0.25in} \mbox{ and } \hspace{0.25in} 
{\mathcal B} = B(0,R) \setminus {\mathcal N}.
\end{displaymath}
We will estimate $\int_{\mathcal N} |Eg(x)|^p H(x) dx$ by induction and
$\int_{\mathcal B} |Eg(x)|^p H(x) dx$ by using the bilinear estimate.

By (\ref{indhyp}) and (\ref{boundoneg}),
\begin{eqnarray*}
\int_{\mathcal N} |Eg(x)|^p H(x) dx 
& \leq & 2^{p-1} \Big( \frac{10}{9} \Big)^p \int_N 
         \sum_{\tau \in S(x)} |Ef_\tau(x)|^p H(x) dx \\
& \leq & \Big( \frac{20}{9} \Big)^p \int_N \sum_\tau |Ef_\tau(x)|^p H(x) dx 
         \\
& \leq & 3^p \sum_\tau C R^\epsilon A_\alpha(H)^\beta R^{\alpha/2} 
         \| f_\tau \|_{L^p(\sigma)}^p \\
&   =  & 3^p C R^\epsilon A_\alpha(H)^\beta R^{\alpha/2} 
         \| g \|_{L^p(\sigma)}^p.
\end{eqnarray*}

To every $x \in {\mathcal B}$ there are two caps $\tau_x, \tau_x' \in S(x)$ 
so that $\mbox{Dist}(\tau_x,\tau_x') \geq \delta$. Writing 
\begin{displaymath}
|Eg(x)|^p= |Eg(x)|^{p/2} |Eg(x)|^{p/2} 
\leq (10 K |Ef_{\tau_x}(x)|)^{p/2} (10 K |Ef_{\tau_x'}(x)|)^{p/2},
\end{displaymath}
we see that
\begin{displaymath}
|Eg(x)|^p \leq (10 K)^p 
\sum_{\tau, \tau': \, \mbox{\tiny Dist}(\tau,\tau') \geq \delta}
|Ef_\tau(x)|^{p/2} |Ef_{\tau'}(x)|^{p/2}.
\end{displaymath}
Using the bilinear estimate (\ref{bilinearest}), it follows that
\begin{eqnarray*}
\lefteqn{\int_{\mathcal B} |Eg(x)|^p H(x) dx} \\
& \leq & (10 K)^p 
       \sum_{\tau, \tau': \, \mbox{\tiny Dist}(\tau,\tau') \geq \delta}
       \int_{\mathcal B} |Ef_\tau(x)|^{p/2} 
    	                 |Ef_{\tau'}(x)|^{p/2} H(x) dx \\
& \leq & (10 K)^p C_B R^\epsilon A_\alpha(H)^\beta R^{\alpha/2}
         \sum_{\tau, \tau': \, \mbox{\tiny Dist}(\tau,\tau') \geq \delta}
         \| f_\tau \|_{L^p(\sigma)}^{p/2} 
		 \| f_{\tau'} \|_{L^p(\sigma)}^{p/2} \\
& \leq & (10)^p K^p C_B R^\epsilon A_\alpha(H)^\beta R^{\alpha/2}
         \sum_{\tau, \tau': \, \mbox{\tiny Dist}(\tau,\tau') \geq \delta}
         \| g \|_{L^p(\sigma)}^{p/2} \| g \|_{L^p(\sigma)}^{p/2}.
\end{eqnarray*}
Therefore,
\begin{displaymath}
\int_{\mathcal B} |Eg(x)|^p H(x) dx
\leq (10)^p K^{p+2} C_B A_\alpha(H)^\beta R^{\alpha/2}
     \| g \|_{L^p(\sigma)}^p.	
\end{displaymath}

Combining the narrow and broad estimates, we arrive at (\ref{indest}).

\underline{The iteration:} In the inductive step we proved that if we were
given that the estimate
\begin{displaymath}
\int_{B(0,R)} |Ef(x)|^p H(x) dx 
\leq C R^\epsilon A_\alpha(H)^\beta R^{\alpha /2} \| f \|_{L^p(\sigma)}^p
\end{displaymath}
holds for every function $f \in L^1(\sigma)$ that is supported in an arc of 
$\sigma$-measure $\leq \delta$, and $\delta$ obeys (\ref{baselinear}), then 
we can produce the estimate
\begin{displaymath}
\int_{B(0,R)} |Eg(x)|^p H(x) dx 
\leq C' R^\epsilon A_\alpha(H)^\beta R^{\alpha /2} \| g \|_{L^p(\sigma)}^p
\end{displaymath}
for every function $g \in L^1(\sigma)$ that is supported in an arc of 
$\sigma$-measure $\leq R^\epsilon \delta$, where
\begin{displaymath}
C'= 3^p C + (10)^p R^{(p+2)\epsilon} C_B.
\end{displaymath}

Starting with the base of the induction, where $\delta = R^{-1/2}$ and 
$C=C_L$, and applying the inductive step $k$ times, we arrive at an estimate 
that holds for every function $f \in L^1(\sigma)$ that is supported on an 
arc of $\sigma$-measure 
$\leq \delta_k = R^{k \epsilon}\delta = R^{k \epsilon}/\sqrt{R}$, with 
constant
\begin{displaymath}
C_k= 3^{kp} C_L + (10)^p R^{(p+2)\epsilon} C_B \sum_{l=0}^{k-1} 3^{lp}
= 3^{kp} C_L + (10)^p R^{(p+2)\epsilon} C_B \frac{1-3^{kp}}{1-3^p}.
\end{displaymath}
At the step before the last, $k=(1/(2\epsilon)) -2$ and
$\delta_k= R^{[(1/(2\epsilon)) -2] \epsilon}/\sqrt{R}=R^{-2\epsilon}$, which 
is a valid value of $\delta$ (i.e.\ $\delta_k=R^{-2\epsilon}$ obeys 
(\ref{baselinear}), because  
$10 R^\epsilon \leq 1/R^{-2\epsilon} \leq R^{1-2\epsilon}/10$). Applying the 
inductive step one last time, we get the estimate
\begin{displaymath}
\int_{B(0,R)} |Ef(x)|^p H(x) dx 
\leq C R^\epsilon A_\alpha(H)^\beta R^{\alpha /2} \| f \|_{L^p(\sigma)}^p
\end{displaymath}
for every function $f \in L^1(\sigma)$ that is supported on an arc of 
$\sigma$-measure $\leq R^{-\epsilon}$, where the constant $C$ satisfies
\begin{displaymath}
C \leq 3^{p/(2\epsilon)} 
\Big( C_L + \frac{(10)^p R^{(p+2)\epsilon}}{3^p-1} C_B \Big).
\end{displaymath}

Since the circle $\mbb S^1$ can be covered by $\sim R^\epsilon$ such arcs,
(\ref{mainest}) follows and Theorem \ref{mainjj} is proved.

\end{document}